\documentclass[12pt]{amsart}

\usepackage{amssymb, amsmath, setspace}
\usepackage[all,cmtip]{xy}
\usepackage{amscd}
\usepackage[bookmarks, bookmarksopen=true, bookmarksnumbered=true]{hyperref}
\usepackage[active]{srcltx}

\theoremstyle{plain}
\newtheorem{theorem}{Theorem}[section]
\newtheorem{prop}[theorem]{Proposition}

\newtheorem{lm}[theorem]{Lemma}

\newtheorem{cor}[theorem]{Corollary}
\newtheorem{conj}[theorem]{Conjecture}
\theoremstyle{definition}

\newtheorem{defn}[theorem]{Definition}
\newtheorem{rmk}[theorem]{Remark}

\newcommand{\ul} \underline

\newcommand{\mtx}{\left[ \begin{matrix}}
\newcommand{\mtxend}{\end{matrix}\right]}

\newcommand{\Lift}{\mbox{Lift}}
\newcommand{\Quad}{\mbox{Quad}}

\DeclareMathOperator{\reg}{{reg}}

\newcommand{\ffi}{\varphi}

\newcommand{\wb}{\widetilde{\beta}}
\def\ZZ{\mathbb Z}
\def\PP{\mathbb P}
\newcommand{\mif}{\mbox{if} ~}

\newcommand{\cP}{\mathcal P}

\begin{document}

\title[Properties of cut ideals associated to ring graphs]{Properties
of cut ideals associated to ring graphs}

\author[Uwe Nagel and Sonja Petrovi\'c]{Uwe Nagel${}^*$ and Sonja Petrovi\'c}
\address{Department of Mathematics, University of Kentucky, Lexington,
KY 40506, USA \\ Department of Mathematics, Statistics and Computer
Science, University of Illinois, Chicago, IL, 60607 USA} \email{{\tt
uwenagel@ms.uky.edu, petrovic@math.uic.edu}}

\thanks{${}^*$  The work for this paper was done while the first
author was sponsored by the National Security Agency under Grant
Number H98230-07-1-0065.
}


\begin{abstract}
A cut ideal of a graph records the relations among the cuts of the graph.  These toric ideals have been introduced by Sturmfels and
Sullivant who also posed the problem of relating their properties to
the combinatorial structure of the graph.\\
We study the cut ideals of the family of ring graphs, which includes
trees and cycles. We show that they have quadratic Gr\"obner
bases and that  their coordinate rings are Koszul,  Hilbertian, and
Cohen-Macaulay, but not  Gorenstein in general.\\
\end{abstract}

\maketitle

\vskip0.2in
\section{Introduction}

Let $G$ be any finite (simple) graph with vertex set $V(G)$ and edge
set $E(G)$. In \cite{StSu} Sturmfels and Sullivant  associate a
projective variety $X_G$ to $G$ as follows.  Let $A|B$ be an
unordered partition of the vertex set of $G$.  Each such partition
defines a cut of the graph, denoted by $Cut(A|B)$, which is the set
of edges $\{i,j\}$ such that $i\in A$, $j\in B$ or $j\in A$, $i\in
B$.  For each $A|B$, we can then assign variables to the edges
according to whether they are in $Cut(A|B)$ or not. The coordinates
$q_{A|B}$ are indexed by the unordered partitions $A|B$, and the
variables encoding whether the edge $\{i, j\}$ is in the cut are
$s_{ij}$ and $t_{ij}$ (for "separated" and "together").  The variety
$X_G$, which we call \emph{the cut variety of $G$}, is specified by
the following homomorphism between polynomial rings:
\begin{align*}
\phi_G : K[q_{A|B} : A|B \mbox{ partition}] &\to K[s_{ij},t_{ij}: \{i,j\}\mbox{ edge of } G], \\
q_{A|B} &\mapsto \prod_{\{i,j\}\in Cut(A|B)}{s_{ij}} \prod_{\{i,j\}\in E(G)\backslash Cut(A|B)}{t_{ij}}
\end{align*}
The cut ideal $I_G$ is the kernel of the map $\phi_G$. It is a homogeneous toric ideal (note that $\deg\phi_G(q_{A|B})=|E(G)|$). The variety $X_G$ is defined
by the cut ideal $I_G$.

Cut ideals generalize toric ideals arising in phylogenetics and the  study of contingency tables. However, the algebraic properties of cut ideals are largely unknown. It is clear that the properties of the cut ideal
depend on the combinatorics of the graph.  Sturmfels and Sullivant pose the following conjecture.

\begin{conj}[\cite{StSu}, Conjecture 3.7.] \label{conj-CM}
The semigroup algebra $K[q]/I_G$ is normal if and only if $K[q]/I_G$ is Cohen-Macaulay if and only if $G$ is free of $K_5$
minors.
\end{conj}

We provide evidence for this conjecture by establishing it for a large class of such graphs. We refer to Section \ref{sec:sp-graphs} for the definition of ring graphs.

\begin{theorem}
  \label{thm-intro}
The cut ideal of a ring graph admits a squarefree quadratic Gr\"obner basis. Hence, its coordinate ring is Cohen-Macaulay.
\end{theorem}

This and further results on cut ideals of ring graphs are
established in Section \ref{sec:sp-graphs}.

The next section is a
collection of necessary definitions and results that we  use
repeatedly. In particular, we recall the theorem of Sturmfels and
Sullivant \cite{StSu} about clique sums. Section \ref{sec-cycles} is
devoted to cycles. Using the correspondence between cut ideals
of cycles and certain phylogenetic ideals on claw trees, we
provide a squarefree quadratic Gr\"obner bases for these cut ideals.
Moreover, we find a  formula for their number of minimal
generators. In Section \ref{sec-trees} we turn to cut
ideals of trees. By \cite{StSu},  the algebraic properties of the cut
ideal of a tree depend only on the number of edges. We determine its $h$-vector.
Section \ref{sec-unions} complements the study of clique sums by
providing a generating set and a Gr\"obner basis for the cut ideal
of a disjoint union of two graphs. The results about cut ideals of
ring graphs are established in Section \ref{sec:sp-graphs}. In
particular, we provide an estimate of their Castelnuovo-Mumford
regularity.

\section{Clique sums, Segre products, and Gr\"obner bases}
\label{sec-background}

We recall some concepts and results that we use later on.
Throughout this paper all graphs are assumed to be finite and simple. The vertex and the edge set of such a graph $G$ are denoted by $V(G)$ and $E(G)$, respectively. A clique of $G$ is a subset of $V(G)$ such that the vertex-induced subgraph of $G$ is complete, that is, there is an edge between any two vertices.

Let now $G_1 = (V_1, E_1)$ and $G_2 = (V_2, E_2)$ be two graphs such that $V_1 \cap V_2$ is a clique of both graphs. Then the {\em clique sum} of $G_1$ and $G_2$ is the graph $G = G_1 \# G_2$ with vertex set $V_1 \cup V_2$ and edge set $E_1 \cup E_2$. If the clique $V_1 \cap V_2$ consists of $k+1$ vertices, then $G$ is also called the {\em $k$-sum} of $G_1$ and $G_2$.
If $0 \leq k \leq 2$, then Sturmfels and Sullivant \cite{StSu} relate the graph-theoretic operation of forming clique  sums to the algebraic operation of taking toric fiber products as defined in \cite{Su}. Defining two operations, Lift and Quad, they show that the generators of the cut ideal $I_G$ can be obtained from the generators of the cut ideals $I_{G_1}$ and $I_{G_2}$.  More precisely, their result is:

\begin{theorem}[\cite{StSu}, Theorem 2.1]
\label{thm-lift-quad}
Let $G$ be a $k$-sum of $G_1$ and $G_2$ with  $0 \leq k \leq 2$. Denote by $\bf{F_1}$ and $\bf{F_2}$  binomial generating sets for the smaller cut ideals
$I_{G_1}$ and $I_{G_2}$. Then
\begin{align*}
\bf{M} = \Lift(\bf{F_1}) \cup \Lift(\bf{F_2}) \cup \Quad(G_1,G_2)
\end{align*}
is a generating set for the cut ideal $I_G$.   Furthermore, if $\bf{F_1}$ and $\bf{F_2 }$ are Gr\"obner bases, then there exists a term order such
that $\bf{M}$ is a Gr\"obner basis of $I_G$.
\end{theorem}

Let $X \subset \PP^n$ be a projective subvariety over a field $K$. We denote its homogeneous coordinate ring by $A_X$. It is a standard
graded $K$-algebra. The variety $X$ is said to be arithmetically Cohen-Macaulay if $A_X$ is a Cohen-Macaulay ring. The Hilbert
function of $X$ or $A_X$ is defined by $h_X (j) = h_{A_X} (j) = \dim_K [A_X]_j$. If $j \gg 0$, then it becomes polynomial in $j$.
This polynomial is the Hilbert polynomial $p_X = p_{A_X}$ of $X$ or $A_X$. Following \cite{Abh}, the variety $X$ (or its coordinate ring
$A_X$) is said to be {\em Hilbertian} if its Hilbert function is polynomial in every non-negative degree, i.e., for every integer $j
\geq 0$, $h_X (j) = p_X (j)$.

The cut ideal of a graph $G$ on $n$ vertices defines a projective variety $X_G \subset \PP^{2^{n-1}-1}$. If $G$ is the $0$-sum of
$G_1$ and $G_2$, then its cut variety $X_G$ is isomorphic to the Segre product of $X_{G_1} \times X_{G_{2}}$. Algebraically, this
means that the coordinate ring $A_{X_G}$ is the Segre product $A_{X_{G_1}} \boxtimes  A_{X_{G_2}}$ of $A_{X_{G_1}}$ and
$A_{X_{G_2}}$. The Segre product of Cohen-Macaulay rings is often not Cohen-Macaulay. The precise result is (see, e.g., \cite{St-V},
Theorem I.4.6):

\begin{lm}
  \label{lem-Segre-CM}
Let $A, B$ be two graded Cohen-Macaulay $K$-algebras that both have dimension at least two. Then their Segre product $A\boxtimes B$ is
Cohen-Macaulay if and only if $A$ and $B$ are Hilbertian.
\end{lm}

Cut ideals are examples of toric ideals. In this note a toric  ideal $I$ is any ideal that  defines an affine or projective variety $X$
that is parametrized by a set of monomials. Thus, $I$ can be generated by a set of binomials. The following result is well-known to  specialists. We provide a proof for the convenience of the reader.

\begin{lm}
  \label{lem-G-and-CM}
Let $I$ be a homogeneous toric ideal of a polynomial ring $S$. If, for  some monomial order on $S$, the initial ideal of $I$ is
squarefree, then $S/I$ is Cohen-Macaulay and Hilbertian.
\end{lm}

\begin{proof}
The ideal  $I$ defines an affine toric variety $Y$ and a projective toric variety $X$. $X$ is projectively normal if and only if $Y$ is
normal. Theorem 13.15 in  \cite{St} says that, if for some term order $\prec$ the initial ideal $in_\prec(I)$ is squarefree, then
$X$ is projectively normal.

By assumption, $S/I$ is isomorphic to a semigroup ring $K[B]$ for  a semigroup $B \subset \mathbb N_0^d$. Thus, using Proposition 13.5 in
\cite{St}, we conclude that the semigroup $B$ is normal. Hence, by a theorem of Hochster \cite{Ho}, the semigroup ring $K[B]$ is
Cohen-Macaulay. Furthermore, in non-negative degrees its Hilbert function is equal to the Ehrhart polynomial of the corresponding polytope (see, e.g., \cite{MS}, Lemma 12.4), thus $S/I$ is Hilbertian.
\end{proof}

\section{Cut ideals of cycles}
\label{sec-cycles}

The starting point of this section is the realization of the cut ideals associated to cycles as certain phylogenetic ideals. This
will enable us to use the main result from \cite{ChPe}.

Recall that the claw
tree $K_{1,n}$ is the complete bipartite graph with $n$ edges from one vertex (the root) to the other $n$ vertices (the leaves). We
denote by $I_n$  the ideal of phylogenetic invariants for the general group-based model for the group $\mathbb Z_2$ on the claw
tree $K_{1,n}$, as in \cite{ChPe}. This ideal  is the kernel of the following homomorphism between polynomial  rings (see \cite{StEtAl}):
\begin{align*}
\varphi_n: R := K[q_{g_1,\dots,g_n}:g_1,\dots,g_n\in \mathbb Z_2]
& \to K[a_g^{(i)}: g\in \mathbb Z_2,i=1,\dots,n+1] =: R'  \\
q_{g_1,\dots,g_n}  & \mapsto
a_{g_1}^{(1)}a_{g_2}^{(2)}\dots a_{g_n}^{(n)}a_{g_1+g_2+\dots +g_n}^{(n+1)}.
\end{align*}
The coordinate $q_{g_1,\dots,g_n}$ corresponds to observing  the element $g_1$ at the first leaf of the tree, $g_2$ at the second,
and so on, though here we are considering the phylogenetic ideals in Fourier coordinates instead of probability coordinates (cf.\
\cite{StSu-05}).

Denote by $C_{n+1}$ the $(n+1)$-cycle on the vertex set $[n+1] := \{1,\ldots,n+1\}$.
We are ready to state the announced comparison result.

\begin{lm}\label{Cycle-Tree-Corr-Lm} The phylogenetic ideal $I_n$ on the claw tree with $n$ leaves and the cut ideal $I_{C_n}$ of an $(n+1)$-cycle agree up to renaming the
variables.
\end{lm}

\begin{proof}
This follows from Theorem 5.5 in \cite{StSu}.
\end{proof}

Recall that each toric ideal has a Gr\"obner basis consisting of binomials. We say that an ideal has a {\em squarefree Gr\"obner basis} if it has a
 Gr\"obner basis consisting of binomials, where each binomial is a difference of squarefree monomials.
 
Combining Lemma \ref{Cycle-Tree-Corr-Lm} and  \cite{ChPe},
Proposition $3$, we obtain the following consequence.

\begin{prop}
\label{prop-cycle}
For each integer $n \geq 4$,  there is an order on the variables such that
the cut ideal of the $n$-cycle has a quadratic  squarefree
Gr\"obner basis with respect to the resulting lexicographic order. In particular, the initial ideal of the cut ideal with respect to this order is squarefree.
\end{prop}

\begin{proof}
  In case $K  = {\mathbb C}$, Proposition 3 in \cite{ChPe} gives the   analogous results for the phylogenetic ideal $I_{k-1}$ on the claw
  tree with $k-1$ edges. The arguments of the proof are valid over an   arbitrary field. Hence, the claim follows by Lemma
  \ref{Cycle-Tree-Corr-Lm}.
\end{proof}

\begin{rmk}
  By \cite{StSu}, Corollary 2.4, the cut varieties defined by $n$-cycles are not smooth if $n \geq 4$.
\end{rmk}

Invoking Lemma \ref{lem-G-and-CM}, we obtain our first contribution to Conjecture
\ref{conj-CM}.

\begin{cor}
The cut variety defined by  any cycle is arithmetically Cohen-Macaulay.
\end{cor}

\begin{rmk}
The ideal $I_{C_4}$ is Gorenstein because it is a complete intersection cut out by three quadrics. However, the cut ideals $I_{C_5}$ and $I_{C_6}$ are not
Gorenstein. We suspect that $I_{C_n}$ is not Gorenstein whenever $n \geq 5$.
\end{rmk}

We conclude this section by determining the number of minimal generators of the cut ideals of cycles. As preparation, we establish a recursion.
Denote by $A_n$ the homogeneous coordinate ring of the variety defined by the phylogenetic ideal $I_n$, that is,
\[
A_n := K[a_{g_1}^{(1)}a_{g_2}^{(2)}\dots a_{g_n}^{(n)}a_{g_1+g_2+\dots +g_n}^{(n+1)} : g_i \in \mathbb Z_2,i=1,\dots,n].
\]
Assigning each variable degree one, as a $K$-algebra $A_n$ is generated  by $2^n$ monomials of degree $n+1$. We claim:

\begin{lm}
\label{lem-recursion} If $n\geq 2$, then $ \dim_K [A_n]_{2(n+1)} = \dim_K [A_{n-1}]_{2n} + 3^n - 2^{n-1}.$
\end{lm}

\begin{proof}
The following set of monomials is a $K$-basis of $[A_n]_{2(n+1)}$:
\begin{align*}
\mathcal M := \{ a_{g_1}^{(1)}a_{h_1}^{(1)}\dots a_{g_n}^{(n)}a_{h_n}^{(n)} a_{g_1+\dots+g_n}^{(n+1)}a_{h_1+\dots+h_n}^{(n+1)} : g_j,h_j\in\mathbb Z_2\}.
\end{align*}
We are going to compare this basis with the set of monomials
\[
{\mathcal N} := \{a_{g_1}^{(1)}a_{h_1}^{(1)}\dots a_{g_n}^{(n)}a_{h_n}^{(n)}  : g_j,h_j\in\mathbb Z_2\}.
\]
Observing that, for the monomials in ${\mathcal M}$, the variables
$a_{g_1+\dots+g_n}^{(n+1)}a_{h_1+\dots+h_n}^{(n+1)}$ are determined by the remaining variables,
one it tempted to guess that there is a bijection between ${\mathcal M}$ and ${\mathcal N}$. However, this is not quite true. To illustrate this,
consider the following example, where $n=3$. Then, by interchanging the factors $a_{0}^{(1)}$ and $a_{1}^{(1)}$ in
\begin{align*}
m_1:=a_{0}^{(1)}a_{0}^{(2)}a_{0}^{(3)}a_{1}^{(1)}a_{1}^{(2)}a_{0}^{(3)} = a_{1}^{(1)}a_{0}^{(2)}a_{0}^{(3)}a_{0}^{(1)}a_{1}^{(2)}a_{0}^{(3)}=:m_2
\end{align*}
this monomial in ${\mathcal N}$ produces two different monomials in $\mathcal M$, namely
\begin{align}
\label{even}\tag{E} m_1 a_{0+0+0}^{(n+1)}a_{1+1+0}^{(n+1)}&= m_1 a_{0}^{(n+1)}a_{0}^{(n+1)} \in [A_3]_8, \\ \label{odd}\tag{O} m_2
a_{1+0+0}^{(n+1)}a_{0+1+0}^{(n+1)}&= m_2 a_{1}^{(n+1)}a_{1}^{(n+1)}
\in
 [A_3]_8.
\end{align}
To keep track if a monomial in ${\mathcal N}$ gives rise to one or two monomials in ${\mathcal M}$ we use the following decomposition
\[
{\mathcal N} = {\mathcal N}_1 \sqcup {\mathcal N}_2 \sqcup {\mathcal N}_3,
\]
where
\begin{eqnarray*}
  {\mathcal N}_1 & := & \{n \in {\mathcal N} : g_1 + \cdots g_n + h_1 + \cdots h_n = 1\}   \\
{\mathcal N}_2 & := &    \{n \in {\mathcal N} : n\ \text{is a square}\} \\
{\mathcal N}_3 & := & \{n \in {\mathcal N} : g_1 + \cdots g_n + h_1 + \cdots h_n = 0, n\ \text{is not a square}\}
\end{eqnarray*}
The monomials in ${\mathcal N}_1$ and ${\mathcal N}_2$ give rise to just one monomial in ${\mathcal M}$, whereas the monomials in ${\mathcal N}_3$
produce precisely two monomials in ${\mathcal M}$ by interchanging the factors $a_{g_i}^{(i)}$ and $a_{h_i}^{(i)}$, where $g_i \neq h_i$.
This interchange  alters the parity of $g_1 + \cdots + g_n = h_1 + \cdots + h_n$. It follows that
\begin{eqnarray*}
  |{\mathcal M}| = |{\mathcal N}_1| + |{\mathcal N}_2| + 2 |{\mathcal N}_3|  = |{\mathcal N}| + |{\mathcal N}_3'|,
\end{eqnarray*}
where
\[
{\mathcal N}_3' := \{n \in {\mathcal N} : g_1 + \cdots g_n = h_1 + \cdots h_n = 0, n\ \text{is not a square}\}.
\]
Each monomial in ${\mathcal N}$ is the product of $n$ quadratic monomials of the form $a_{g_i}^{(i)} \cdot  a_{h_i}^{(i)}$, where $g_i, h_i \in \ZZ_2$.
For fixed $i$, there are three such monomials, thus we get $|{\mathcal N}| = 3^n$, hence
\[
|{\mathcal M}| = 3^n + |{\mathcal N}_3'|.
\]
Notice that the condition $g_1 + \cdots g_n = h_1 + \cdots h_n = 0$ is equivalent to $g_1+\dots+g_{n-1} = g_n$ and $h_1 + \cdots h_{n-1} = h_n$.
It follows that the dimension of $[A_{n-1}]_{2n}$ equals the sum of $|{\mathcal N}_3'|$ and the number of squares $n \in {\mathcal N}$ with
$g_1 + \cdots + g_{n-1} = g_n$. Since, there are $2^{n-1}$ such squares, we obtain
$ \dim_K [A_{n}]_{2(n+1)} = |{\mathcal M}| = 3^n + \dim_K [A_{n-1}]_{2n} - 2^{n-1}$ ,
as claimed.
\end{proof}

Now we are ready to compute the  number of minimal generators of a
cut ideal associated to a cycle. Possibly, it is not too surprising
that it  has a nice   combinatorial interpretation.

Recall that the Stirling number $S(n,k)$ of the second kind is the number of partitions of a set with $n$ elements into $k$ blocks (cf. \cite{Stan-1999}, page
33). Note that $S(n, 4) = \frac{1}{24}(4^n-4 \cdot 3^n+6 \cdot 2^n-4)$.

\begin{prop}
\label{prop-gen-cut}
If $n\geq 1$, then the cut ideal of an $n+1$-cycle is minimally generated by
 $3 \cdot S(n+1, 4)$ quadratic binomials.
\end{prop}

\begin{proof}
By Lemma \ref{Cycle-Tree-Corr-Lm}, it is equivalent to compute the number of minimal generators of the phylogenetic ideal $I_n$ corresponding to the
claw tree $K_{1,n}$. Using the above notation, $I_n$ is an ideal in the polynomial ring $R$ with $2^n$ variables.
We first compute the Hilbert function of the quotient ring $R/I_n$ in degree 2, that is, $h_n(2) := \dim_K [R/I_n]_2$. Since $I_n = \ker \ffi_n$, we
get $h_n (2) = \dim_K [A_n]_{2(n+1)}$. Hence, Lemma \ref{lem-recursion} gives if $n \geq 2$:
\[
h_n(2) = h_{n-1} (2) + 3^n - 2^{n-1}.
\]
Using $h_1(2) = 3$, it follows that
\begin{eqnarray*}
  h_n (2) = h_1 (2) + \sum_{i=2}^n 3^n - \sum_{i=1}^{n} 2^n = \frac{1}{2} [3^{n+1} - 1] - [2^n - 1] = \frac{3}{2} 3^n - 2^n + \frac{1}{2}.
\end{eqnarray*}
Since, by Proposition \ref{prop-cycle}, the ideal $I_n$ is generated in degree two, its number $\mu (I_n)$ of  minimal generators is
\begin{eqnarray*}
  \mu (I_n) = \dim_K [R]_2 - h_n (2) = \binom{2^n + 1}{2} - \frac{3}{2} 3^n + 2^n - \frac{1}{2} = \frac{1}{2} 4^n - \frac{3}{2} 3^n  + \frac{3}{2} 2^n - \frac{1}{2}.
\end{eqnarray*}
It is easily checked that the last number equals $3 \cdot S(n+1, 4)$. The proof is complete.
\end{proof}


\section{Cut ideals of trees}
\label{sec-trees}

As observed by Sturmfels and Sullivant in \cite{StSu}, the algebraic properties of the cut ideal associated to a tree only depend on the number of edges and not on the specific
structure of the tree. Here we use this to determine the $h$-vector. It turns out that its entries admit combinatorial interpretations.

Our starting point is:

\begin{theorem}
\label{thm-trees}
Let $T$ be a tree with $n \geq 1$ edges, and let $X_T \subset {\mathbb P}^{2^n-1}$ be the toric variety defined by the cut ideal $I_T$.
Then $X_T$ is arithmetically Gorenstein of dimension $n$ and degree $n!$. More precisely, $X_T$ is isomorphic to the Segre embedding of
$(\mathbb P^1)^n$ into $\mathbb P^{2^n-1}$ and its Hilbert function is
\begin{align*}
  h_{S_T/I_T} (i) = (i+1)^n  \quad (i \geq 0).
  \end{align*}
\end{theorem}

\begin{proof}
Since each tree with $n \geq 2$ edges can be obtained as the zero sum of an edge and a tree with $n-1$ edges, the cut variety $X_T$ is isomorphic to the claimed Segre embedding as shown in
\cite[Example 2.3]{StSu}. This also implies that $X_T$ is arithmetically Gorenstein (\cite{Hoa}, Corollary 3.3).

The Hilbert function of a Segre product is the product of the Hilbert functions of the factors. This gives the claim on the Hilbert function of $X_T$.
\end{proof}

Again, we see that the number of minimal generators of a cut ideal  grows rapidly when the number of edges increases.

\begin{cor}
  \label{cor-gen-tree}
If\; $T$ is a tree with $n$ edges, then its cut ideal is minimally generated by $2 \cdot 4^{n-1} + 2^{n-1} - 3^n$ quadrics.
\end{cor}

\begin{proof}
  The cut ideal $I_T$ lies in a polynomial ring $S_n$ with $2^n$ variables. Using that,  by the above theorem,  $h_{S_n/I_T} (2) = 3^n$  the claim follows.
\end{proof}

The  cut ideal of a tree has a minimal generating set that is even a Gr\"obner basis. Note
that the cut ideal of a tree with one edge is trivial.

\begin{cor}
If\;  $T$ is a tree with a least two edges, then there is a monomial order such that its cut ideal has a quadratic squarefree Gr\"obner
basis. In particular, the corresponding initial ideal is squarefree.
\end{cor}

\begin{proof}
  This follows by Theorem \ref{thm-lift-quad}  as the Lift and Quad operations preserve the squarefree structure and degree of the binomials. (See also Corollary 18 of \cite{Su}.)
\end{proof}

Our next goal is to make the Hilbert series of the cut variety of a tree explicit. We will see  that it admits a combinatorial interpretation.

Recall that the Hilbert series of  a standard graded $K$-algebra $B$ is the formal power series
\[
H_B (t) = \sum_{i \geq 0} h_B (i) t^i.
\]
It is a rational function that can be uniquely written as
\[
H_B (t) = \frac{h_0 + h_1 t + \cdots + h_s t^s}{(1-t)^d},
\]
where $d$ is the (Krull) dimension of $B$,\ $h_0 = 1,h_1,\cdots,h_s$ are integers, and $h_s \neq 0$. If $B$ is Cohen-Macaulay, then all $h_i$ are
non-negative and $(h_0,\cdots,h_s)$ is called the $h$-vector of $B$. For a projective subscheme $X \subset \PP^N$, these concepts are defined by using
its homogeneous coordinate ring $A_X$. However, the {\em Castelnuovo-Mumford regularity} of $X$ is defined as the regularity of its ideal sheaf. Thus,
$\reg X = \reg A_X + 1$.

For a positive integer $n$, denote by ${\mathcal S}_n$  the symmetric group on $n$ letters. The $n$-th {\em Eulerian polynomial} $A_n$ is defined as
\[
A_n (t)  := \sum_{\sigma \in {\mathcal S}_n} t^{1 + d(\sigma)},
\]
where $d(\sigma)$ is the number of descents of the permutation $\sigma$ (see, e.g., \cite{Stan-1999}, page 22). Writing
\[
A_n (t) = A_{n, 1} t + \cdots + A_{n, n} t^n,
\]
the coefficients $A_{n, k}$ are called {\em Eulerian numbers}. Like binomial coefficients, they satisfy a recurrence relation:
\begin{align*}
A_{n,i+1} = (n-i) A_{n-1,i} + (i+1) A_{n-1,i+1}.
\end{align*}

For trees, the above concepts are related:

\begin{prop}
\label{prop-Hilb-tree}
Let $T$ be a tree with $n \geq 1$ edges. Then the Hilbert series of its cut variety $X_T$ is
\begin{align*}
 H_{X_T} (t)=  \frac{A_{n,1} + A_{n,2} t+ \dots + A_{n,n}t^{n-1} }{ (1-t)^{n+1} }.
\end{align*}
Moreover, the Castelnuovo-Mumford regularity is $\reg X_T = n.$
\end{prop}

\begin{proof}
By Theorem \ref{thm-trees}, we know that the Hilbert series of $X$ is
\[
H_X (t) = \sum_{i\geq 0}{(i+1)^n t^i}.
\]
It is known (see, for example,  \cite{Stan-1999}, page 209) that the Eulerian polynomials satisfy
\begin{align*}
\sum_{i\geq 0} i^nt^i = \frac{A_n(t)}{(1-t)^{n+1}}.
\end{align*}
Dividing by $t$ provides the desired formula for the Hilbert series of $X$.

Finally, since $X$ is arithmetically Cohen-Macaulay, the regularity of its homogeneous coordinate ring is the degree of the numerator polynomial in the
Hilbert series.  Hence $A_{n,n}= 1$ provides $\reg X  = n$.
\end{proof}

As $X_T$ is arithmetically Gorenstein, its $h$-vector is symmetric, that is $ A_{n,k} = A_{n,n+1-k}$. This also follows directly form the interpretation of the
Eulerian number $A_{n,k}$ as the number of permutations in ${\mathcal S}_n$ with $k-1$ excedances
(see \cite{Stan-1999}, Proposition 1.3.12).

\section{Disjoint unions}
\label{sec-unions}

We want to show that the cut ideal of a disjoint union of two graphs can be studied by means of their zero-sum.

We need a general fact:

\begin{lm}
  \label{lem-doubling}
Let $R= K[x_1,\ldots,x_n]$ and $S := K[x_1,\ldots,x_n, y_1,\ldots,y_n]$ be polynomial rings in $n$ and $2n$ variables, respectively.
Let $\psi: R \to T$ be any $K$-algebra homomorphism and consider the homomorphism $\ffi: S \to T$ that is defined by
$\ffi (y_i) = \ffi (x_i) := \psi (x_i)$, \; $i = 1,\ldots,n$. Set $I := \ker \ffi$ and $J := \ker \psi$.
Then
\[
I = J \cdot S + (x_1 - y_1,\ldots,x_n - y_n)
\]
and  $S/I \cong R/J$.

Moreover, if, for some monomial order on $R$, ${\bf F}$ is a Gr\"obner basis of $J$, then  ${\bf F} \cup \{x_1-y_1,\ldots,x_n -
y_n\}$ is a Gr\"obner basis of $I$ with respect to some monomial order on $S$.
\end{lm}

\begin{proof}
  This is probably well-known to  specialists. For the convenience   of the reader we provide a short proof.
The $K$-algebra homomorphism $\gamma: S \to R$ that maps $x_i$ and $y_i$ onto $x_i$ induces the following commutative diagram with
exact rows and column
\begin{equation*}
\begin{CD}
& & & &  0 \\
& & & &  @VVV \\
& & & & L \\
& & & &  @VVV \\
0 @>>> I @>>>S @>{\ffi}>> T \\
& & @VVV @VV{\gamma}V  @VV{=}V\\
0 @>>> J @>>> R @>{\psi}>> T,  \\
& & & &  @VVV \\
& & & &  0 \\
\end{CD}
\end{equation*}
where $L$ is the ideal $L := (x_1 - y_1,\ldots,x_n - y_n)$. Thus, we
get an exact sequence
\[
0 \to L \to I \to J \to 0.
\]
Using also the natural embedding of $R$ as a subring of $S$, we conclude that $I = J \cdot S + L$, as desired.

The claim about the Gr\"obner bases follows by using an elimination order on $S$ that extends the term order on $R$ used for computing
the Gr\"obner basis ${\bf F}$ with the property that each variable $y_1 > y_2 > \cdots > y_n$ is greater than any monomial in $R$.
\end{proof}

Let $G_1$ and $G_2$ be (non-empty) graphs on $v_1$ and $v_2$ vertices.  Consider the zero-sum $G_0:=G_1 \# G_2$ obtained by
joining the two graphs at any vertex. Its cut ideal lies in a polynomial ring $R$ with $2^{v_1+v_2 - 2}$ variables. The disjoint
union of the two graphs $G_\sqcup :=G_1 \sqcup G_2$ defines a cut ideal in a polynomial ring $S$ in $2^{v_1+v_2 - 1}$ variables. Using
this notation, the main result of this section is:

\begin{prop}
\label{prop-disjoint}
There is an injective, graded $K$-algebra homomorphism $\alpha: R \to S$ mapping the variables of $R$ onto variables of $S$ such that
\[
I_{G_1 \sqcup G_2} = \alpha( I_{G_0}) + L,
\]
where $L \subset S$ is an ideal that is minimally generated by $2^{v_1+v_2 - 2}$ linear forms.

Furthermore, the cut variety $X_{G_1 \sqcup G_2} \subset \PP^{2^{v_1+v_2-1}-1}$ is isomorphic to the Segre embedding of  $X_{G_1} \times X_{G_2}$
into $\PP^{2^{v_1+v_2-2}-1}$.
\end{prop}

\begin{proof}
To simplify notation, set $G_\sqcup := G_1 \sqcup G_2$. Moreover, denote the polynomial rings that are used to define cut ideals of $G_\sqcup$ and the
zero sum $G_0$ by $S, S', R, R'$, that is, $I_{G_\sqcup}$ is the kernel of $\ffi_{G_\sqcup}: S \to S'$ and $I_{G_0}$ is the kernel of $\ffi_{G_0}: R \to R'$.

Let $x \in V(G_1)$ and $y \in V(G_2)$ be the vertices of $G_1$ and $G_2$ that are identified in the $0$-sum $G_0$. It will be convenient to denote the
resulting vertex in $G_0$ by $z$.

There is a natural bijection $\wb: E(G_0) \to E(G_\sqcup)$, defined by
\[
\{i, j\} \mapsto \wb (\{i, j\}) := \left \{
\begin{array}{ll}
  \{i, j\} & \mif z \notin \{i, j\} \\
\{i, x\} & \mif j = z, i \in G_1 \\
\{i, y\} & \mif j = z, i \in G_2
\end{array}.
\right.
\]
It induces an isomorphism $\beta: R' \to R$.

Now consider any unordered
partition $A|B$ of the vertex set of $G_0$. We may assume that $z \in A$. Then we define a partition $A'|B$ of the vertex set of $G_\sqcup$ by setting
$A' := (A \setminus \{z\}) \cup \{x, y\}$. This induces an injective $K$-algebra homomorphism $\alpha: R \to S$ that maps the variable $q_{A|B} \in R$
onto the corresponding variable $q_{A'|B} \in S$.

Observing that $\wb$ maps $Cut (A|B)$ onto $Cut (A'|B)$, we get a commutative diagram
\begin{equation*}
\begin{CD}
R @>{\ffi_{G_{0}}}>> R' \\
 @VV{\alpha}V  @VV{\beta}V\\
S @>{\ffi_{G_{\sqcup}}}>> S'.  \\
\end{CD}
\end{equation*}
Since $\alpha$ is injective, it follows immediately that $J := \alpha (I_{G_0}) \subset I_{G\sqcup}$.

We now consider the set $\cP$ of partitions of the vertex set of $G_\sqcup$. We decompose it as \[
\cP = \cP_1 \sqcup \cP_2,
\]
where
\[
\cP_1 := \{ A|B \in \cP : x, y \in A\}
\]
and
\[
\cP_2 := \{A|B \in \cP : x\in A, \ y \in B\}
\]
Given a partition  $A|B \in \cP_1$, define sets $C, D$:
\begin{eqnarray*}
  C & := & (A \cap V(G_1)) \cup (B \cap V(G_2)) \\[1ex]
D  & :=  & (B \cap V(G_1)) \cup (A \cap V(G_2)).
\end{eqnarray*}
Then $C|D$ is a partition in $\cP_2$, thus we get a map $\varepsilon: \cP_1 \to \cP_2$.
Note that this map is bijective, thus $|\cP_1| = \frac{1}{2} |\cP| = 2^{v_1 + v_2-1}$.
Moreover, the cut sets of $A|B \in \cP_1$ and $C|D = \varepsilon (A|B)$ are the same.
Using also that $\alpha$ maps the variables in $R$ onto variables in $S$ indexed by partitions in $\cP_1$, Lemma \ref{lem-doubling} provides
\begin{equation}
\label{eq-disjoint}
  I_{G_\sqcup} = J + (q_{A|B} - q_{\varepsilon(A|B)} : A|B \in \cP_1)
\end{equation}
and an isomorphisms between the homogeneous coordinate rings of the cut varieties defined by $G_\sqcup$ and $G_0$.
Since taking $0$-sums corresponds to forming Segre products, it follows that $X_{G_\sqcup} \cong X_{G_1} \times X_{G_2}$, and the proof is complete.
\end{proof}

The above proof also implies:

\begin{cor}
\label{cor-Gr-disjoint-u}
If the  cut ideals of $G_1$ and $G_2$ admit a squarefree Gr\"obner basis, then so does the cut ideal of their disjoint union.
\end{cor}

\begin{proof}
The assumption implies that the cut ideal of the zero sum of $G_1$ and $G_2$ admits a squarefree Gr\"obner basis. Hence, using the second assertion of
Lemma \ref{lem-doubling}, Equation (\ref{eq-disjoint}) provides the claim.
\end{proof}

Now we address the transfer of the Cohen-Macaulay property under forming disjoint unions.

\begin{cor}
  \label{cor-unions}
Let $G_1$ and $G_2$ be two graphs such  that that their cut varieties are arithmetically Cohen-Macaulay. Then the cut variety
associated to the disjoint union of $G_1$ and $G_2$ is arithmetically Cohen-Macaulay if and only if the two varieties
defined by $G_1$ and $G_2$ are  Hilbertian.
\end{cor}

\begin{proof}
Denote by $T, T_1$, and $T_2$ the homogeneous coordinate rings  of the cut varieties associated to $G_1 \sqcup G_2, G_1$, and $G_2$,
respectively. Proposition \ref{lem-doubling} provides that $T$ is isomorphic to the Segre product $T_1 \boxtimes T_2$. If, say, $G_1$
does not have any edge, then $T_1$ is isomorphic to the a polynomial ring in one variable. Thus it is Hilbertian and $T$ is isomorphic to $T_2$.

If both graphs $G_1$ and $G_2$ have at least one edge, the Krull dimension of $T_1$ and $T_2$ is at least two. Then the claim is a
consequence of Lemma \ref{lem-Segre-CM}.
\end{proof}

Recall that a {\em forest} is a disjoint union of trees. Thus we get:

\begin{cor}
  \label{cor-forest}
The cut variety defined by any forest is arithmetically Gorenstein.
\end{cor}

\begin{proof}
This follows immediately by  combining Theorem \ref{thm-trees} and Proposition \ref {prop-disjoint}.
\end{proof}


\section{Cut ideals of ring graphs}
\label{sec:sp-graphs}

We are ready to analyze cut ideals of more complicated graphs using our results for cut ideals of trees and cycles. In order to describe the graphs, we need additional vocabulary.

A vertex $v$ of a graph $G$ is called a \emph{cutvertex} if the number of connected components of the vertex-induced subgraph on $V(G) \setminus \{v\}$
is larger than that of $G$.
Similarly, an edge $e$ is called a \emph{bridge} if the number of connected components of $G\backslash \{e\}$ is larger than that of $G$.
A \emph{block} of $G$ is a maximal connected subgraph of $G$ without cut vertices.
It follows that each block of a graph is either an isolated vertex, a bridge, or a maximal $2$-connected subgraph.

\begin{defn}[\cite{VillEtAl}, \cite{VillEtAl2}]  \rm
    A \emph{ring graph} is a graph $G$ with the property that each block of $G$ which is not a bridge or a vertex can be constructed from a cycle
    by successively adding cycles of length at least $3$ using the edge-sum ($1$-sum) construction.
\end{defn}
Gitler, Reyes and Villarreal study the family of graphs whose number of primitive cycles equals its cycle rank.  They show that this family is precisely the family of ring graphs. Ring graphs were first introduced in \cite{VillEtAl} and \cite{VillEtAl2}. They are a subclass of series-parallel graphs. For further background on series-parallel graphs, the reader should refer to \cite{Di}.

Examples of ring graphs include trees and cycles.  More precisely, ring graphs are those graphs that can be obtained from
trees and cycles by performing clique sums over vertices or edges. They form a large subclass of series parallel graphs, since they are free
of $K_4$-minors (\cite{VillEtAl}).

Combining our results from the previous sections, we obtain the following more precise version of Theorem \ref{thm-intro}:
\begin{theorem}\label{thm:main-tree+cycle}
    If $G$ is a ring graph, 
    then the cut variety $X_G$ is generated by quadrics.
    In addition, there exists a term order for which its defining ideal $I_G$ has a squarefree quadratic Gr\"obner basis.

    Therefore, such varieties $X_G$ are Hilbertian and arithmetically Cohen-Macaulay, but not  arithmetically Gorenstein in general.
\end{theorem}
\begin{proof}
    The theorem follows by applying Proposition \ref{prop-cycle}, Theorem \ref{thm-trees}, Corollary \ref{cor-Gr-disjoint-u},
    and Theorem \ref{thm-lift-quad} repeatedly.
\end{proof}
Recall that if an ideal $I$ admits a quadratic Gr\"obner basis, then the coordinate ring $S/I$ is Koszul (see \cite{Anick}). Thus, we get.

\begin{cor} \label{cor:tree+cycle=koszul}
    The coordinate ring of the cut variety associated to an arbitrary ring graph is Koszul.
\end{cor}

\begin{rmk}\rm
    Theorem 1.3 of \cite{StSu} characterizes those graphs whose cut ideals have squarefree reverse-lexicographic initial ideals.
    Arbitrary ring graphs \emph{do not} fall into that category, however our result shows that they \emph{do} have squarefree initial ideals
    with respect to another term order.
\end{rmk}

We conclude this note with an estimate on the degrees of the higher
syzygies of cut ideals. In general, the Castelnuovo-Mumford
regularity  of a homogeneous ideal can grow doubly exponentially in
the number of variables. However, the following result shows that
cut ideals of ring graphs have much better properties also in this
respect. Recall that the cut ideal of a graph with $n$ vertices
lives in a polynomial ring in $2^{n-1}$ variables.

\begin{cor}
  \label{cor-reg}
If $G$ is a ring graph with $e$ edges, then the Castelnuovo-Mumford regularity of the cut variety $X_G$ satisfies
\[
\reg X_G \leq e + 1.
\]
\end{cor}

\begin{proof}
The least integer $r$ such that the  Hilbert  function and the Hilbert polynomial of a graded algebra in each degree $j \geq r$ is often called the index of
regularity. By Theorem \ref{thm:main-tree+cycle} we know that the coordinate ring $A_G$ of any cut variety associated to  a ring graph $G$ is Hilbertian,
i.e., its index of regularity is at most zero. However, in case of Cohen-Macaulay rings the index of regularity and the Castelnuovo-Mumford regularity are
related (see, e.g., \cite[Lemma 8]{torino}). It follows that $ \reg X_G - 1 = \reg A_G < \dim A_G.$
 Since the dimension of the  ring $A_G$ is one plus the
number of edges of $G$, the argument is complete.
\end{proof}

\begin{rmk}
  \label{rem-reg}
(i) Note that by Proposition \ref{prop-Hilb-tree} the above bound is almost sharp if the graph is a tree.

(ii)
The above arguments use only the fact that cut ideals of ring graphs admit a squarefree Gr\"obner basis. Thus, they provide the following observation:

If the cut ideal of a graph $G$ admits a squarefree Gr\"obner basis, then the Castelnuovo-Mumford regularity of the cut variety $X_G$ is bounded above by one plus the number of edges of $G$.
\end{rmk}

\section*{Acknowledgement}
The authors would like to thank Rafael Villarreal for introducing
them to ring graphs.



\begin{thebibliography}{25}

\bibitem{Abh}
S. Abhyankar: {\emph{Enumerative Combinatorics of Young Tableaux}},
Marcel Dekker, New York, 1988.

\bibitem{Anick}
D. Anick: {\emph{On the homology of associative algebras}},  Trans.\
Amer.\ Math.\ Soc.\ {\bf 296}  (1986),  641--659.

\bibitem{ChPe}
J. Chifman, S. Petrovi\'c: {\emph{Toric ideals of phylogenetic
invariants for the general group-based model on claw trees}}, in:
Proceedings of the Second international conference on  Algebraic
Biology, (eds. H. Anai, K. Horimoto and T. Kutsia), Springer LNCS
{\bf 4545}, Springer-Verlag, 307-321, 2007.

\bibitem{Di}
 R. Diestel: \emph{Graph Theory}, 2nd ed.,
 Graduate Texts in Mathematics {\bf{173}},  Springer-Verlag, New York,  2000.

\bibitem{StEtAl}
N. Eriksson, K. Ranestad, B. Sturmfels, S. Sullivant:
{\emph{Phylogenetic algebraic geometry}}, in:  Projective Varieties
with Unexpected Properties,  (eds. C. Ciliberto, A. Geramita, B.
Harbourne, R-M. Roig and K. Ranestad), De Gruyter, Berlin, 2005, pp.
237--255.

\bibitem{VillEtAl}
I. Gitler, E. Reyes, R. Villarreal: \emph{Ring graphs and complete
intersection toric ideals}, Discrete Mathematics (to appear).

\bibitem{VillEtAl2}
I. Gitler, E. Reyes, R. Villarreal: \emph{Ring graphs and toric
ideals}, Electronic Notes in Discrete Mathematics, 28C (2007),
393--400.

\bibitem{Hoa}
L. T. Hoa: \emph{On Segre products of affine semigroup rings}, Nagoya Math.\ J.\
{\bf 110} (1988), 113--128.

\bibitem{Ho}
M. Hochster: {\emph{Rings of invariants of tori, Cohen-Macaulay
rings generated by monomials, and polytopes}}, Ann.\ Math. {\bf 96}
(1972), 318--337.

\bibitem{torino}
J.\ Migliore, U. Nagel: \emph{Liaison and Related Topics: Notes from the Torino    Workshop/School}, Rend.\ Sem.\ Mat.\ Univ.\
     Politec.\ Torino {\bf 59} (2003), 59--126.

\bibitem{MS}
E. Miller and B. Sturmfels: \emph{Combinatorial
Commutative Algebra}, Graduate Texts in Mathematics {\bf 227},
Springer-Verlag, New York, 2005.


\bibitem{Stan-1999}
R. Stanley:  {\emph{Enumerative Combinatorics}}, Volume I, Cambrigde Studies in Advanced Mathematics {\bf 49}, Cambridge University Press, New York, 1999.

\bibitem{St-V}
J. St\"uckrad, W. Vogel: {\emph{Buchsbaum Rings and Applications.
An Interaction between Algebra, Geometry and Topology}},
Springer-Verlag, Berlin, 1986.

\bibitem{St}
B. Sturmfels:  {\emph{Gr\"{o}bner Bases and Convex Polytopes}},
 University Lecture Series {\bf 8}, American Mathematical Society, Providence, 1996.

\bibitem{StSu-05}
B. Sturmfels, S. Sullivant:  {\emph{Toric ideals of phylogenetic
invariants}}, J.\ Comp.\ Biol.\ {\bf 12} (2005), 204--228.


\bibitem{StSu}
B. Sturmfels, S. Sullivant: {\emph{Toric geometry of cuts and
splits}}, Michigan Math.\ J.\ {\bf 57} (2008), 689--709.

\bibitem{Su}
S. Sullivant: \emph{Toric fiber products},  J.\ Algebra {\bf 316}
(2007), 560--577.

\bibitem{4ti2}
4ti2 team:  4ti2 -- A software package for algebraic, geometric  and
combinatorial problems on linear spaces.  Available at www.4ti2.de .

\end{thebibliography}
\end{document}